\documentclass[11pt,reqno]{amsart}

\usepackage[T1]{fontenc}
\usepackage[utf8]{inputenc}
\usepackage{lmodern}
\usepackage[a4paper,margin=27mm]{geometry}
\usepackage{amsmath,amssymb,mathtools}
\usepackage{booktabs}
\usepackage[expansion=false]{microtype}
\usepackage{xcolor}
\usepackage[colorlinks=true,linkcolor=blue!55!black,citecolor=blue!55!black,urlcolor=blue!55!black]{hyperref}
\hypersetup{pdftitle={Complex slices, inverses, and deformations of a three-dimensional real algebra},
pdfauthor={Alejandro Jes\'us Bermejo Vald\'es},
pdfsubject={},pdfkeywords={},pdfcreator={}}

\numberwithin{equation}{section}

\newtheorem{theorem}{Theorem}[section]
\newtheorem{proposition}[theorem]{Proposition}

\newtheorem{corollary}[theorem]{Corollary}
\theoremstyle{definition}
\newtheorem{definition}[theorem]{Definition}
\newtheorem{remark}[theorem]{Remark}

\DeclareMathOperator{\Der}{Der}
\DeclareMathOperator{\Aut}{Aut}
\DeclareMathOperator{\End}{End}
\DeclareMathOperator{\Cent}{Cent}
\newcommand{\B}{\mathcal A}
\newcommand{\R}{\mathbb R}
\newcommand{\C}{\mathbb C}
\newcommand{\one}{\mathbf 1}
\newcommand{\spanR}{\operatorname{span}_{\R}}
\newcommand{\assoc}[3]{(#1,#2,#3)}

\title[Complex slices, inverses, and deformations]
{Complex slices, inverses, and deformations of a three-dimensional real algebra}

\author{Alejandro Jes\'us Bermejo Vald\'es}
\address{Riojan Health Service, Piqueras 98, 26006 Logro\~no, La Rioja, Spain}
\email{ajbermejo@riojasalud.es}


\begin{document}

\begin{abstract}
We study the real three-dimensional unital algebra with basis
$\{\one,i,j\}$ and products
\[
 i^2=j^2=-\one,\qquad ij=0,\qquad ji=i+j.
\]
Its quadratic and power-associative elements form the union of exactly three complex
subalgebras; their nonzero elements are precisely those with a two-sided
inverse, which is unique. Every nonzero subalgebra contains the unit, and
every element outside these three planes generates the whole algebra.
We classify one-sided inverses, including the singular cases, and show
that two-sided invertibility does not imply cancellation. The multiplication
algebra is $M_3(\R)$, so the algebra is absolutely central simple.
Its derivations vanish, its automorphism group is trivial, and its commutator
Lie algebra is $\R\oplus\mathfrak{aff}(1)$. Its unique antiautomorphism is
involutive, and its complexification admits no complex-linear structurable
involution. Finally, replacing $ji=i+j$
by $ji=t(i+j)$ gives real algebras $\B_t$ with
$\B_s\cong\B_t$ exactly when $s^2=t^2$. We describe the exceptional
parameters and thereby exhibit nontrivial deformations of the original algebra.
\end{abstract}

\maketitle
\enlargethispage{3pt}

\section{Introduction}

We study a three-dimensional real unital algebra in which two-sided
invertibility does not imply cancellation. Its multiplication table also
provides explicit failures of power-associativity and multiplicativity of
the Euclidean quadratic form.

The main feature is the coincidence of three loci: quadratic elements,
power-associative elements, and, after excluding zero, elements with a
two-sided inverse. They are governed by the polynomial
$bc(b-c)=\det[\one,x,x^2]$, where $x=a\one+bi+cj$.
The same polynomial detects when one element generates the whole algebra.
This gives an explicit setting in which to distinguish inversion,
cancellation, and associativity of powers.

We determine the remaining structure directly from the product. We first
compute the associator and the nuclei.  We then prove simplicity by showing
that the multiplication algebra is the full endomorphism algebra.  The
remaining sections classify the associative two-dimensional subalgebras,
two-sided inverses, one-sided zero divisors, quadratic elements, derivations,
automorphisms, and involutions.  We identify the commutator Lie algebra and
construct an explicit deformation family, distinguishing the absence of
symmetries from rigidity of the multiplication.

Lie-admissibility itself is automatic for every three-dimensional unital
real algebra, as proved in Proposition~\ref{prop:automatic-lie}; it is not a
special identity distinguishing this multiplication. The classification of
complex three-dimensional unital structurable algebras with involution in
\cite{AbdurasulovEraliyevaKaygorodov2026} concerns a more restricted class:
the complexification of the present algebra admits no complex-linear
structurable involution (Proposition~\ref{prop:anti-involution}).
The present work is a structural case study; historical priority
for the multiplication table and the family is not established here.

The member $\B_0$ of the family defined in Section~\ref{sec:deformations}
has $i^2=j^2=-\one$ and $ij=ji=0$. It is the classical Jordan algebra
of the nondegenerate symmetric bilinear form $-\langle\,,\,\rangle$ on $\R^2$
\cite[Chapter IV, Section 2, example D, p.~101]{Schafer1966}.
In $\B_0$, $i(-i+\lambda j)=(-i+\lambda j)i=\one$ for every
$\lambda\in\R$, while $ij=0$. Thus existence of a two-sided inverse
without cancellation already occurs in a classical example. The algebra
$\B=\B_1$ combines this failure with uniqueness of the two-sided inverse,
exactly three complex planes, and trivial automorphisms. The results below
concern this combination and the classification of the displayed family.

We use the standard terminology of finite-dimensional nonassociative
algebras~\cite{Schafer1966}.  Lie-admissible and Malcev-admissible algebras are
understood in the usual sense; see
\cite{LauferTomber1962,GozeRemm2004,Myung1986}.

Throughout the paper the ground field is $\R$.  Scalar extensions will be
introduced explicitly when needed.  All products are written by juxtaposition,
and no reassociation is made unless it has been justified.

\section{Definition and elementary identities}

\begin{definition}\label{def:B}
Let $\B$ be the real vector space with ordered basis $\{\one,i,j\}$.  Define a
bilinear multiplication by the table
\begin{equation}\label{eq:table}
\begin{array}{c|ccc}
     &\one&i&j\\ \hline
\one &\one&i&j\\
i    &i&-\one&0\\
j    &j&i+j&-\one
\end{array}.
\end{equation}
Equivalently, after identifying $a\one+bi+cj$ with $(a,b,c)$,
\begin{equation}\label{eq:product}
(a,b,c)(d,e,f)
=\bigl(ad-be-cf,\;ae+bd+ce,\;af+ce+cd\bigr).
\end{equation}
\end{definition}

Because~\eqref{eq:product} is linear in each argument, it defines an
$\R$-algebra.  Substitution of $(1,0,0)$ on either side shows that $\one$ is a
two-sided identity.  Thus~\eqref{eq:table}, without any additional convention,
is a complete definition of a unital nonassociative algebra.

For $x,y,z\in\B$, write
\[
\assoc{x}{y}{z}:=(xy)z-x(yz)
\]
for the associator.  The left, middle, and right nuclei are respectively
\begin{align*}
N_{\ell}(\B)&=\{x:\assoc{x}{\B}{\B}=0\},\\
N_m(\B)&=\{y:\assoc{\B}{y}{\B}=0\},\\
N_r(\B)&=\{z:\assoc{\B}{\B}{z}=0\}.
\end{align*}
The nucleus is their intersection, and the center is the set of nuclear
elements commuting with every element of $\B$.

\begin{proposition}[Associator formula]\label{prop:associator}
For
\[
x=(a,b,c),\qquad y=(d,e,f),\qquad z=(g,h,r),
\]
the associator is
\begin{equation}\label{eq:associator}
\assoc{x}{y}{z}=
\begin{pmatrix}
bfh-ceh-cer+cfh\\
bfr+ceh-2cfh\\
-ber+2ceh-cfh
\end{pmatrix}.
\end{equation}
In particular,
\begin{equation}\label{eq:third-power}
\assoc{x}{x}{x}=bc(b-c)(i+j).
\end{equation}
\end{proposition}

\begin{proof}
Using~\eqref{eq:product}, first compute
\[
xy=(ad-be-cf,\;ae+bd+ce,\;af+ce+cd)
\]
and
\[
yz=(dg-eh-fr,\;dh+eg+fh,\;dr+fh+fg).
\]
Applying~\eqref{eq:product} once more to $(xy)z$ and $x(yz)$, subtracting
componentwise, and collecting coefficients gives~\eqref{eq:associator}.
Putting $(d,e,f)=(g,h,r)=(a,b,c)$ then yields
\[
\assoc{x}{x}{x}=(0,bc(b-c),bc(b-c)),
\]
which is~\eqref{eq:third-power}.
\end{proof}

\begin{corollary}\label{cor:identity-failures}
The algebra $\B$ is noncommutative, nonassociative, nonalternative,
nonflexible, and not power-associative.
\end{corollary}

\begin{proof}
Noncommutativity follows from $ij=0$ and $ji=i+j$.  Moreover,
\[
\assoc{i}{j}{i}=\one,
\]
so $\B$ is not associative. Left alternativity fails because
$\assoc{i}{i}{j}=-j\neq0$. Right alternativity fails because
\[
(ij)j=0\neq i(jj)=-i,
\]
and flexibility fails because
\[
(ji)j=-\one\neq j(ij)=0.
\]
Finally, for $x=i+2j$, formula~\eqref{eq:third-power} gives
$\assoc{x}{x}{x}=-2(i+j)\neq0$.  Hence the two possible products of three
copies of $x$ disagree, so $\B$ is not power-associative.
\end{proof}

\begin{proposition}[Nuclei and center]\label{prop:nuclei}
One has
\[
N_{\ell}(\B)=N_m(\B)=N_r(\B)=Z(\B)=\R\one.
\]
\end{proposition}

\begin{proof}
Let $n=(a,b,c)$.  Substitution in~\eqref{eq:associator} gives
\[
\assoc{n}{i}{j}=(-c,0,-b),\qquad
\assoc{i}{n}{j}=(0,c,-b),\qquad
\assoc{i}{j}{n}=(b,c,0).
\]
Thus membership in any one of the three nuclei forces $b=c=0$.  Conversely,
scalar multiples of the unit associate with all pairs, so all three nuclei
are $\R\one$.  Their intersection is therefore $\R\one$, and every such
element commutes with $\B$; hence the center is also $\R\one$.
\end{proof}

\section{Multiplication algebra and absolute central simplicity}

For $x\in\B$, let $L_x(y)=xy$ and $R_x(y)=yx$.  The \emph{multiplication
algebra} $\mathcal M(\B)$ is the associative subalgebra of $\End_{\R}(\B)$
generated by all $L_x$ and $R_x$.  Relative to $\{\one,i,j\}$,
\begin{equation}\label{eq:LR}
L_{(a,b,c)}=
\begin{pmatrix}
a&-b&-c\\ b&a+c&0\\ c&c&a
\end{pmatrix},\qquad
R_{(a,b,c)}=
\begin{pmatrix}
a&-b&-c\\ b&a&b\\ c&0&a+b
\end{pmatrix}.
\end{equation}

Recall that a two-sided ideal $I$ of $\B$ satisfies $\B I+I\B\subseteq I$.
The centroid is
\[
\Cent(\B)=\{T\in\End_{\R}(\B):T(xy)=T(x)y=xT(y)
\text{ for all }x,y\in\B\}.
\]

\begin{theorem}\label{thm:multiplication}
The multiplication algebra of $\B$ is the full endomorphism algebra:
\[
\mathcal M(\B)=\End_{\R}(\B)\cong M_3(\R).
\]
Consequently, $\B$ is simple and $\Cent(\B)=\R\,\mathrm{id}_{\B}$.  More
strongly, for every field extension $F/\R$, the scalar extension
$\B_F=\B\otimes_{\R}F$ is simple and has centroid
$F\,\mathrm{id}_{\B_F}$.  Thus $\B$ is absolutely central simple.
\end{theorem}

\begin{proof}
Consider the following nine elements of $\mathcal M(\B)$:
\begin{equation}\label{eq:nine-ops}
I,\ L_i,\ L_j,\ R_i,\ R_j,\ L_i^2,\ L_iL_j,\ L_jL_i,\ L_iR_i.
\end{equation}
For a $3\times3$ matrix $X=(x_{uv})$, let
\[
\operatorname{vec}(X)=(x_{11},x_{12},x_{13},x_{21},x_{22},x_{23},
x_{31},x_{32},x_{33})^{\mathsf T}.
\]
Using~\eqref{eq:LR}, the matrix whose columns are the vectorizations of the
operators in~\eqref{eq:nine-ops}, in the displayed order, is
\begin{equation}\label{eq:Q}
Q=\begin{pmatrix}
1&0&0&0&0&-1&0&0&-1\\
0&-1&0&-1&0&0&-1&0&0\\
0&0&-1&0&-1&0&0&0&-1\\
0&1&0&1&0&0&0&1&0\\
1&0&1&0&0&-1&0&0&-1\\
0&0&0&1&0&0&-1&0&0\\
0&0&1&0&1&0&0&1&0\\
0&0&1&0&0&0&0&-1&0\\
1&0&0&1&0&0&0&0&0
\end{pmatrix},
\qquad \det Q=-1.
\end{equation}
Thus the nine operators are linearly independent.  Since
$\dim_{\R}\End_{\R}(\B)=9$, they form a basis of $\End_{\R}(\B)$, proving
$\mathcal M(\B)=\End_{\R}(\B)$.

Every two-sided ideal is invariant under every $L_x$ and $R_x$, hence under
$\mathcal M(\B)$.  If $0\neq I\triangleleft\B$, choose $0\neq v\in I$.  For
an arbitrary $w\in\B$, there is an endomorphism $T$ with $T(v)=w$.  Since
$T\in\mathcal M(\B)$, invariance gives $w\in I$; hence $I=\B$.

An element of the centroid commutes with all left and right multiplications,
and therefore with $\mathcal M(\B)=\End_{\R}(\B)$.  The commutant of the full
matrix algebra in itself consists exactly of scalar matrices, so
$\Cent(\B)=\R\,\mathrm{id}_{\B}$.

Finally, the calculation~\eqref{eq:Q} remains valid after extension of
scalars: its determinant is still $-1\in F$.  Hence
$\mathcal M(\B_F)=\End_F(\B_F)$, and the preceding ideal and centroid
arguments apply verbatim over $F$.
\end{proof}

\begin{remark}[Operators and algebra representations]
The map $x\mapsto L_x$ is injective since $L_x\one=x$, but it does not
preserve multiplication: $L_iL_j(i)=-\one$ while $L_{ij}(i)=0$.
Thus the matrices represent multiplication operators, not the product of
$\B$ as an associative matrix product. Indeed, an injective
product-preserving map from $\B$ into any associative algebra would force
all associators of $\B$ to vanish.
\end{remark}

\begin{remark}
Simplicity alone does not imply a division property, even for associative
algebras: full matrix algebras already provide counterexamples.  Here
$ij=0$ with $i,j\neq0$, although $\B$ is simple.  The specifically
nonassociative phenomenon in Section~\ref{sec:inverses} is that a two-sided
invertible element can nevertheless annihilate a nonzero element.
In particular, the identification of the three planes with $\C$ in
Theorem~\ref{thm:slices}, the real subalgebra classification in
Corollary~\ref{cor:all-subalgebras},
the inverse classification in Theorem~\ref{thm:inverse}, the real solution
sets in Propositions~\ref{prop:quadratic} and~\ref{prop:roots-minus-one},
and the real isomorphism types and subalgebra counts in
Proposition~\ref{prop:exceptional} depend on the ground field.
For example, in $\B\otimes_{\R}\C$, if $\iota^2=-1$ is a scalar, then
$(\one+\iota i)/2$ is a nontrivial idempotent.
The symmetry results in Theorem~\ref{thm:symmetries} and
Proposition~\ref{prop:anti-involution}, however, hold over every extension
field of $\R$, as stated there.
\end{remark}

\section{Associative slices and the power-associative locus}

An element $x$ is called \emph{quadratic} if
$x^2\in\spanR\{\one,x\}$. This convention includes scalar elements.
The subalgebra generated by a set means the smallest linear subspace closed
under multiplication and containing that set; the unit is not added unless
specified. An element is \emph{power-associative} if all products of any
fixed positive number of copies of it agree, irrespective of parenthesization.

Set $k=i+j$ and define
\[
\C_i=\spanR\{\one,i\},\qquad
\C_j=\spanR\{\one,j\},\qquad
\C_k=\spanR\{\one,k\}.
\]

\begin{theorem}\label{thm:slices}
Each of $\C_i,\C_j,\C_k$ is an associative unital subalgebra isomorphic to
$\C$.  They are precisely the two-dimensional subalgebras of $\B$ that
contain $\one$.  Moreover, for $x=(a,b,c)$ the following are equivalent:
\begin{enumerate}
\item the subalgebra generated by $\one$ and $x$ is associative;
\item all powers of $x$ are independent of parenthesization;
\item $\assoc{x}{x}{x}=0$;
\item $bc(b-c)=0$;
\item $x\in\C_i\cup\C_j\cup\C_k$;
\item $x$ is quadratic.
\end{enumerate}
\end{theorem}

\begin{proof}
Any unital two-dimensional algebra with basis $\{\one,v\}$ and relation
$v^2=\alpha\one+\beta v$ is isomorphic, by its multiplication table, to
$\R[X]/(X^2-\beta X-\alpha)$, and is therefore associative.
The relations $i^2=j^2=-\one$ show that $\C_i$ and $\C_j$ are copies of
$\C$.  Since
\begin{equation}\label{eq:k-square}
k^2=(i+j)^2=-2\cdot\one+k,
\end{equation}
$\C_k$ is closed and associative.  More explicitly,
\[
\ell:=\frac{2k-\one}{\sqrt7}
\quad\text{satisfies}\quad \ell^2=-\one,
\]
so $\C_k=\R\one\oplus\R\ell\cong\C$.

Let $S$ be a two-dimensional subalgebra containing $\one$.  Then
$S=\spanR\{\one,v\}$ for some $v=a\one+bi+cj$ not in $\R\one$.  Replacing
$v$ by $v-a\one$, we may assume $v=bi+cj$, with $(b,c)\neq(0,0)$.  Directly,
\begin{equation}\label{eq:v-square}
v^2=-(b^2+c^2)\one+bc(i+j).
\end{equation}
Closure of $S$ is equivalent to $bc(i+j)\in\R v$.  If $bc=0$, this gives
$S=\C_i$ or $S=\C_j$.  If $bc\neq0$, comparison of the $i$- and
$j$-coordinates gives $b=c$, and then $S=\C_k$.  This proves the
classification of unital two-dimensional subalgebras.

The implication $(1)\Rightarrow(2)$ is immediate, and
$(2)\Rightarrow(3)$ follows by comparing $(x^2)x$ with $x(x^2)$.  By
Proposition~\ref{prop:associator}, (3) and (4) are equivalent.  The elementary
factorization in (4) is equivalent to $b=0$, $c=0$, or $b=c$, which is exactly
(5).  Finally, each plane in (5) is associative, so $(5)\Rightarrow(1)$.
For a nonscalar $x$, being quadratic is equivalent to closure of
$\spanR\{\one,x\}$, so the classification above proves $(6)\Leftrightarrow(5)$.
Scalar elements satisfy both conditions.
\end{proof}

\begin{corollary}[All subalgebras and single generators]\label{cor:all-subalgebras}
The complete list of real subalgebras, without requiring them to contain
$\one$ in their definition, is
\[
0,\quad \R\one,\quad \C_i,\quad \C_j,\quad \C_k,\quad \B.
\]
In particular, every nonzero subalgebra contains $\one$.  If $bc(b-c)\neq0$,
then $x=(a,b,c)$ generates $\B$, and $\{\one,x,x^2\}$ is a vector-space basis.
\end{corollary}

\begin{proof}
Let $S$ be a nonzero subalgebra and choose $0\neq x\in S$.  If $x$ lies in
one of the three complex planes, its complex quadratic equation has the form
$x^2-Tx+N\one=0$ with real $T$ and $N>0$.  Hence $\one\in S$.
Otherwise~\eqref{eq:third-power} gives
$(x^2)x-x(x^2)=bc(b-c)k\in S$, so $k\in S$ and
$\one=(k-k^2)/2\in S$.  Any one-dimensional such subalgebra is $\R\one$;
the two-dimensional ones are classified by Theorem~\ref{thm:slices}.
Finally, direct expansion gives
\[
\det[\one\;x\;x^2]
=bc(b-c).
\]
Thus an element outside the three planes generates a subalgebra containing
three linearly independent vectors, which must be $\B$.
More generally, a set generates $\B$ exactly when it is contained in none
of the three planes: every proper subalgebra in the list is contained in
one of them.
\end{proof}

\begin{corollary}[Power-series exponential on the slices]\label{cor:exponentials}
The powers in the usual exponential series are termwise independent of
parenthesization for exactly the elements in
$\C_i\cup\C_j\cup\C_k$.  On those planes it is
\begin{align*}
\exp(a\one+bi)&=e^a(\cos b\,\one+\sin b\,i),\\
\exp(a\one+bj)&=e^a(\cos b\,\one+\sin b\,j),\\
\exp(a\one+bk)&=
e^{a+b/2}\left(
\cos\frac{\sqrt7 b}{2}\,\one
+\sin\frac{\sqrt7 b}{2}\,\frac{2k-\one}{\sqrt7}
\right).
\end{align*}
For $bc(b-c)\neq0$, the coefficient customarily denoted $x^3/3!$ is already
ambiguous.  This prevents a termwise parenthesization-independent power
series there; it does not prevent defining global exponentials by a specified
recursion.
\end{corollary}

\begin{proof}
Theorem~\ref{thm:slices} gives both the necessity and the three associative
copies of $\C$.  The first two formulas are the ordinary complex exponential.
For the third, write
\[
a\one+bk=(a+b/2)\one+(\sqrt7 b/2)\ell
\]
and use $\ell^2=-\one$.  Outside the stated union,
$(x^2)x\neq x(x^2)$ by~\eqref{eq:third-power}.
\end{proof}

\begin{remark}[Global recursive exponentials]
The matrix exponentials
\[
E_L(x)=\exp(L_x)\one,\qquad E_R(x)=\exp(R_x)\one
\]
are defined and convergent for every $x\in\B$.  Their terms use, respectively,
$x(x\cdots x)$ and $(\cdots xx)x$.  They agree with the preceding formulas
on each complex slice.  For real $\tau\to0$,
\[
E_R(\tau x)-E_L(\tau x)
=\frac{\tau^3}{6}\assoc{x}{x}{x}+O(\tau^4).
\]
Consequently they differ for sufficiently small nonzero $\tau$ whenever
$bc(b-c)\neq0$.  Different recursions can therefore yield different globally defined
analytic maps. For an explicit convergence bound, put
$\|(a,b,c)\|_1=|a|+|b|+|c|$. The table gives
$\|uv\|_1\leq2\|u\|_1\|v\|_1$. Induction on a binary product tree bounds
every parenthesization $P_n(x)$ of $n\geq1$ copies of $x$ by
$2^{n-1}\|x\|_1^n$. Hence every fixed choice of parenthesizations gives an
absolutely convergent series $\one+\sum_{n\geq1}P_n(x)/n!$, locally
uniformly in $x$. The two matrix definitions above are real analytic by
the convergent matrix exponential series.
\end{remark}

\section{Two-sided inverses and one-sided zero divisors}\label{sec:inverses}

In a nonassociative algebra, an element $y$ is a \emph{two-sided inverse} of
$x$ if $xy=yx=\one$.  This definition alone does not imply cancellation,
because the expressions $(yx)z$ and $y(xz)$ need not agree.

\begin{theorem}\label{thm:inverse}
A nonzero element $x=(a,b,c)$ has a two-sided inverse if and only if
\[
bc(b-c)=0.
\]
Equivalently, the two-sided invertible locus is
\[
(\C_i\cup\C_j\cup\C_k)\setminus\{0\}.
\]
Each such inverse is unique, and the formulas are
\begin{align}
(a\one+bi)^{-1}&=\frac{a\one-bi}{a^2+b^2},\label{eq:inv-i}\\
(a\one+cj)^{-1}&=\frac{a\one-cj}{a^2+c^2},\label{eq:inv-j}\\
(a\one+bk)^{-1}&=
\frac{(a+b)\one-bk}{a^2+ab+2b^2}.\label{eq:inv-k}
\end{align}
\end{theorem}

\begin{proof}
Suppose $y=(p,q,r)$ satisfies $xy=yx=\one$.  Equality of the second
coordinates in $xy$ and $yx$, using~\eqref{eq:product}, gives
\begin{equation}\label{eq:cqbr}
cq=br.
\end{equation}
The vanishing of the second and third coordinates of $xy$ gives
\begin{equation}\label{eq:inverse-linear}
(a+c)q+bp=0,\qquad ar+cq+cp=0.
\end{equation}
Assume $bc\neq0$.  From~\eqref{eq:cqbr}, $r=cq/b$, and from the first equation
of~\eqref{eq:inverse-linear}, $p=-(a+c)q/b$.  Substitution in the second gives
\[
\frac{c(b-c)}{b}\,q=0.
\]
If $b\neq c$, then $q=0$, and consequently~\eqref{eq:cqbr} and
\eqref{eq:inverse-linear} give $r=p=0$, contradicting the real-coordinate
equation $ap-bq-cr=1$.  Hence $b=c$.  We have proved the necessary condition
$bc(b-c)=0$.

Conversely, if $c=0$ or $b=0$, direct multiplication verifies
\eqref{eq:inv-i} or~\eqref{eq:inv-j}.  If $b=c$, then $x=a\one+bk$ and
\eqref{eq:k-square} verifies~\eqref{eq:inv-k}.  The three denominators are
strictly positive for nonzero real pairs; in the last case,
\[
a^2+ab+2b^2=(a+b/2)^2+\frac74b^2.
\]
Thus the displayed elements are defined and are two-sided inverses.

For uniqueness, suppose first that $(b,c)\neq(0,0)$.  Equation~\eqref{eq:cqbr}
forces $(q,r)$ to be proportional to $(b,c)$, so any two-sided inverse belongs
to $\spanR\{\one,x\}$.  When $x$ satisfies the necessary condition, this
space is one of the three fields above, in which inverses are unique.
For $x=a\one\neq0$, the equation $xy=\one$ immediately gives
$y=a^{-1}\one$.
\end{proof}

\begin{remark}[Failure of cancellation]
The element $i$ has the two-sided inverse $-i$, but $ij=0$ with $j\neq0$.
There is no contradiction: attempting to deduce $j=(-i)(ij)=0$ would replace
$((-i)i)j$ by $(-i)(ij)$, an invalid reassociation.  Thus ``two-sided
invertible'' and ``not a zero divisor'' are distinct notions here.
\end{remark}

We call $x$ a left annihilator of a nonzero element if
$xy=0$ for some $y\neq0$, and a right annihilator of a nonzero element if
$yx=0$ for some $y\neq0$.  Formula~\eqref{eq:LR} gives the exact algebraic
loci.

\begin{proposition}\label{prop:zero-divisors}
For $x=(a,b,c)$, there exists $0\neq y$ with $xy=0$ if and only if
\begin{equation}\label{eq:delta-L}
\Delta_L(x):=a^3+a^2c+a(b^2+c^2)-bc^2+c^3=0.
\end{equation}
There exists $0\neq y$ with $yx=0$ if and only if
\begin{equation}\label{eq:delta-R}
\Delta_R(x):=a^3+a^2b+a(b^2+c^2)+b^3-b^2c=0.
\end{equation}
\end{proposition}

\begin{proof}
The equation $xy=0$ has a nonzero solution exactly when $L_x$ is singular,
and $yx=0$ has a nonzero solution exactly when $R_x$ is singular.  Expanding
the determinants in~\eqref{eq:LR} gives
$\det L_x=\Delta_L(x)$ and $\det R_x=\Delta_R(x)$.
\end{proof}

\begin{remark}[Distinct left and right inverses]
Nonsingularity of both multiplication operators need not give a common
two-sided inverse.  For $x=\one+i+2j$, one has
$\Delta_L(x)=12$ and $\Delta_R(x)=6$.  The unique solutions of $xy=\one$
and $zx=\one$ are, respectively,
\[
y=\tfrac14\one-\tfrac1{12}i-\tfrac13j,
\qquad z=\tfrac13\one-\tfrac13j.
\]
They are different, as also follows from Theorem~\ref{thm:inverse}.
\end{remark}

\begin{proposition}[Complete one-sided inverse criteria]\label{prop:one-sided}
Let $x=(a,b,c)\neq0$. A right inverse, meaning a solution of $xy=\one$,
exists exactly when $\Delta_L(x)\neq0$ or
\[
x\in\R i\cup\R k\cup\R(\one-j).
\]
If $\Delta_L(x)\neq0$, it is unique and equals
\[
\frac{a(a+c)\one-ab i+c(b-a-c)j}{\Delta_L(x)}.
\]
A left inverse, meaning a solution of $yx=\one$, exists exactly when
$\Delta_R(x)\neq0$ or
\[
x\in\R j\cup\R k\cup\R(\one-i).
\]
If $\Delta_R(x)\neq0$, it is unique and equals
\[
\frac{a(a+b)\one+b(c-a-b)i-acj}{\Delta_R(x)}.
\]
In each singular case where an inverse exists, all inverses on that side
form an affine line. Left cancellation by $x$ holds exactly when
$\Delta_L(x)\neq0$, and right cancellation exactly when $\Delta_R(x)\neq0$.
\end{proposition}

\begin{proof}
The regular formulas are the first columns of the adjugate matrices divided
by their determinants. For the singular right-inverse problem, the last
two rows of $L_x$ form
\[
B_x=\begin{pmatrix}b&a+c&0\\c&c&a\end{pmatrix}.
\]
Their cross product is $v=(a(a+c),-ab,c(b-a-c))$.
If $v\neq0$, then $\ker B_x=\R v$. Since the first row of $L_x$ paired
with $v$ is $\Delta_L(x)=0$, no vector in this kernel can satisfy the
first-coordinate equation $ap-bq-cr=1$.
If $v=0$ and $a\neq0$, then $b=0,c=-a$.
If $v=0$ and $a=0$, then $c(b-c)=0$, so $c=0$ or $b=c$.
These are precisely the three displayed lines. Each lies in an associative
complex slice, so Theorem~\ref{thm:inverse} supplies a solution $y_0$.
On each nonzero line $L_x$ has rank two: this follows immediately by
substituting $i,k,\one-j$ into~\eqref{eq:LR}. Thus all solutions are
$y_0+\ker L_x$, an affine line.

For left inverses, let $P(a,b,c)=(a,c,b)$. The coordinate matrices give
$R_x=P L_{Px}P$. This identity also follows from the antiautomorphism
in Proposition~\ref{prop:anti-involution}: $P((Px)(Py))=yx$. Applying the preceding argument to $Px$ and then applying
$P$ gives the left-inverse assertions. Finally, cancellation is injectivity
of the corresponding multiplication operator, equivalent in finite
dimension to a nonzero determinant.
\end{proof}

The affine solution sets from Proposition~\ref{prop:one-sided} are made
explicit in Table~\ref{tab:inverse-lines}. For each row, multiplication
checks the particular inverse and the displayed kernel direction.
Exhaustivity follows from the rank-two calculation in the singular cases
and from nonsingularity of the corresponding multiplication operator in
the remaining cases.
For $x=\rho u$, $\rho\neq0$, divide the particular solution for $u$ by
$\rho$ and retain its kernel direction. In particular, $i$ has a unique
two-sided inverse $-i$, although its right inverses form a whole affine line.

\begin{table}[ht]
\caption{All inverses for representatives of the singular lines.
Each occurrence of $\lambda$ ranges independently over $\R$.}
\label{tab:inverse-lines}
\centering
\begin{tabular}{@{}ccc@{}}
\toprule
$x$ & Right inverses: $xy=\one$ & Left inverses: $yx=\one$ \\
\midrule
$i$ & $-i+\lambda j$ & $-i$ \\
$j$ & $-j$ & $-j+\lambda i$ \\
$k$ & $(\one-k)/2+\lambda(\one-i+j)$ &
$(\one-k)/2+\lambda(\one+i-j)$ \\
$\one-j$ & $(\one+j)/2+\lambda(-\one+2i+j)$ & $(\one+j)/2$ \\
$\one-i$ & $(\one+i)/2$ & $(\one+i)/2+\lambda(-\one+i+2j)$ \\
\bottomrule
\end{tabular}
\end{table}

These failures have minimal dimension among real unital algebras.
Indeed, a unital algebra of dimension two has a basis $\{\one,v\}$ and
$v^2=\alpha\one+\beta v$, so its table is that of the associative algebra
$\R[X]/(X^2-\beta X-\alpha)$. Dimension one is also associative.
In an associative algebra, if $xy=zx=\one$, then
$z=z(xy)=(zx)y=y$; multiplication by a two-sided inverse also gives
cancellation. Hence neither the failures of inversion and cancellation
above nor failure of power-associativity can occur in dimension at most two.

\section{Polynomial equations of degree two}

For $x=(a,b,c)$, direct multiplication gives
\begin{equation}\label{eq:square}
x^2=\bigl(a^2-b^2-c^2,\;b(2a+c),\;c(2a+b)\bigr).
\end{equation}

\begin{proposition}[Idempotents and square-zero elements]\label{prop:quadratic}
The only idempotents of $\B$ are $0$ and $\one$.  The only solution of
$x^2=0$ is $x=0$.
\end{proposition}

\begin{proof}
Suppose first that $x^2=x$.  Equation~\eqref{eq:square} is equivalent to
\begin{equation}\label{eq:idempotent-system}
a^2-b^2-c^2=a,\qquad b(2a+c-1)=0,\qquad c(2a+b-1)=0.
\end{equation}
If $b=c=0$, then $a^2=a$, giving $x=0$ or $x=\one$.  If exactly one of
$b,c$ is nonzero, then $a=1/2$ and the first equation would require the square
of that nonzero coordinate to be $-1/4$, impossible over $\R$.  If $bc\neq0$,
the last two equations give $b=c=1-2a$.  The first then becomes
$7a^2-7a+2=0$, whose discriminant is $-7$.  Thus there are no further real
idempotents.

Now suppose $x^2=0$.  If $b=c=0$, then $a=0$.  If exactly one of $b,c$ is
nonzero, the last two coordinates of~\eqref{eq:square} force $a=0$, while the
first coordinate is the negative of a nonzero square.  If $bc\neq0$, the last
two coordinates give $b=c=-2a$; the first becomes $-7a^2=0$, again forcing
$x=0$.  Hence no nonzero square-zero element exists.
\end{proof}

\begin{proposition}[Square roots of $-\one$]\label{prop:roots-minus-one}
The equation $x^2=-\one$ has exactly six real solutions:
\begin{equation}\label{eq:roots-minus-one}
\mathcal S=\left\{\pm i,\ \pm j,\
\ \pm\frac{\one-2i-2j}{\sqrt7}\right\}.
\end{equation}
\end{proposition}

\begin{proof}
By~\eqref{eq:square}, the equation is
\begin{equation}\label{eq:root-system}
a^2-b^2-c^2=-1,\qquad b(2a+c)=0,\qquad c(2a+b)=0.
\end{equation}
If $c=0$, then $b\neq0$, $a=0$, and $b=\pm1$, giving $\pm i$.  The case
$b=0$ similarly gives $\pm j$.  If $bc\neq0$, then $b=c=-2a$, and the first
equation becomes $-7a^2=-1$.  This gives the remaining two elements in
\eqref{eq:roots-minus-one}.  The case $b=c=0$ has no real solution.
\end{proof}

\section{Derivations, automorphisms, and involutions}

\begin{theorem}\label{thm:symmetries}
For every field extension $F/\R$, put $\B_F=\B\otimes_{\R}F$.
Its $F$-linear derivations and automorphisms satisfy
\[
\Der_F(\B_F)=0,\qquad \Aut_F(\B_F)=\{\mathrm{id}_{\B_F}\}.
\]
In particular, these conclusions hold over both $\R$ and $\C$.
\end{theorem}

\begin{proof}
Work over $F$ and let $D\in\Der_F(\B_F)$.  Since $D(\one)=D(\one^2)=2D(\one)$, one has
$D(\one)=0$.  Write $D(i)=a\one+bi+cj$.  Applying $D$ to $i^2=-\one$ gives
\[
0=D(i)i+iD(i)=-2b\one+(2a+c)i+cj.
\]
Thus $a=b=c=0$ and $D(i)=0$.  Similarly, if
$D(j)=d\one+ei+fj$, then $D(j^2)=0$ gives
\[
0=D(j)j+jD(j)=-2f\one+ei+(2d+e)j,
\]
so $D(j)=0$.  Therefore $D=0$.

Now let $\varphi\in\Aut_F(\B_F)$. It fixes $\one$, since it sends a
two-sided identity to a two-sided identity. Multiplicativity gives
\[
L_{\varphi(x)}=\varphi L_x\varphi^{-1},\qquad
R_{\varphi(x)}=\varphi R_x\varphi^{-1}.
\]
Thus both traces are invariant. By~\eqref{eq:LR},
\[
\operatorname{tr}L_{(a,b,c)}=3a+c,\qquad
\operatorname{tr}R_{(a,b,c)}=3a+b.
\]
Writing $p=\varphi(i)$ and $q=\varphi(j)$, their traces therefore force
\[
p=(a,1-3a,-3a),\qquad q=(d,-3d,1-3d).
\]
The $i$- and $j$-coordinates of $p^2=-\one$ are respectively
$3a^2-a=0$ and $3a^2-3a=0$. Subtracting gives $2a=0$, hence $a=0$.
For $q^2=-\one$ the two coordinates give $3d^2-3d=0$ and
$3d^2-d=0$, whence $d=0$. Consequently $p=i$ and $q=j$, so
$\varphi$ fixes a basis. This argument uses no ordering or positivity in $F$.
\end{proof}

\begin{proposition}[Antiinvolution]\label{prop:anti-involution}
The linear map
\begin{equation}\label{eq:star}
(a,b,c)^*:=(a,c,b)
\end{equation}
extends, for every $F/\R$, to the unique $F$-linear antiautomorphism of
$\B_F$, and it is involutive:
\[
(x^*)^*=x,\qquad (xy)^*=y^*x^*.
\]
Over $\R$, its fixed and skew subspaces are
\[
\mathcal H=\spanR\{\one,i+j\},\qquad
\mathcal S_-=\spanR\{i-j\}.
\]
Over $F$ these are the corresponding $F$-spans.
No $\B_F$ admits an $F$-linear involution making it structurable;
in particular, the complexification admits no complex-linear structurable
involution.
\end{proposition}

\begin{proof}
The identity $(x^*)^*=x$ is immediate from~\eqref{eq:star}.  Since the product is
bilinear, it suffices to verify $(uv)^*=v^*u^*$ for
$u,v\in\{\one,i,j\}$.  The products involving $\one$ are automatic, and the
four remaining checks follow from
\[
(i^2)^*=-\one=j^2,\quad (j^2)^*=-\one=i^2,\quad
(ij)^*=0=ij,\quad (ji)^*=i+j=ji.
\]
The descriptions of $\mathcal H$ and $\mathcal S_-$ follow by solving
$x^*=x$ and $x^*=-x$.
If $\sigma$ is any other $F$-linear antiautomorphism, then $*\circ\sigma$
is an automorphism.  Theorem~\ref{thm:symmetries} gives $*\circ\sigma=\mathrm{id}$,
and hence $\sigma=*$. Thus the failure below rules out every $F$-linear
structurable involution on $\B_F$.

For the last assertion, recall that for a unital algebra with involution one
sets
\[
V_{x,y}(t)=(xy^*)t+(ty^*)x-(tx^*)y,\qquad
T_z(t)=zt+tz-tz^*.
\]
A necessary identity for structurability
\cite[Definition 2, equation (1)]{AbdurasulovEraliyevaKaygorodov2026} is
\begin{equation}\label{eq:structurable}
T_zV_{x,y}-V_{x,y}T_z
=V_{T_z(x),y}-V_{x,T_{z^*}(y)}.
\end{equation}
It is the specialization $u=z$, $v=\one$ of the usual identity
\cite{Allison1978}
\[
[V_{u,v},V_{x,y}]=V_{V_{u,v}(x),y}-V_{x,V_{v,u}(y)},
\]
where the bracket denotes the commutator of operators. Indeed,
$V_{z,\one}=T_z$ and $V_{\one,z}=T_{z^*}$ directly from the definitions.
Failure of this specialization suffices for the conclusion.
Evaluate both sides of~\eqref{eq:structurable} at $t=i$ with
$x=\one$, $y=i$, and $z=i$.  Using only
\eqref{eq:table}, the left-hand side is $-\one+i+4j$, whereas the right-hand
side is $\one-2i-j$.  Their difference is $-2\cdot\one+3i+5j\neq0$, so
\eqref{eq:structurable} fails over every $F/\R$.
\end{proof}

\begin{remark}[Euclidean length is not a composition norm]\label{rem:norm}
The positive-definite quadratic form
\[
Q(a,b,c)=a^2+b^2+c^2
\]
is the usual Euclidean squared length on the underlying vector space, but it
is not multiplicative: $Q(i)=Q(j)=1$ and $Q(ij)=0$.
In fact, $ij=0$ rules out every positive-definite multiplicative quadratic
form, since its values at both $i$ and $j$ would be positive.
Likewise, the reflection
$x^\dagger=(a,-b,-c)$ is not an antiautomorphism, since
$(ij)^\dagger=0$ but $j^\dagger i^\dagger=ji=i+j$.  Directly,
\[
xx^\dagger=Q(x)\one-bc(i+j),
\]
which is not generally scalar.  Thus neither $Q$ nor $\dagger$ supplies a
composition-algebra conjugation.  The genuine antiautomorphism is~$*$ from
Proposition~\ref{prop:anti-involution}, but $xx^*$ is also not generally
scalar.
\end{remark}

\section{The commutator Lie algebra}

Define $[x,y]=xy-yx$.  An algebra is called Lie-admissible if this commutator
is a Lie bracket~\cite{GozeRemm2004}.

\begin{proposition}[Automatic Lie-admissibility in dimension three]
\label{prop:automatic-lie}
Every three-dimensional real algebra with a two-sided unit is Lie-admissible.
\end{proposition}

\begin{proof}
For an anticommutative bilinear bracket the Jacobiator
\[
J(x,y,z)=[x,[y,z]]+[y,[z,x]]+[z,[x,y]]
\]
is alternating and trilinear.
For a commutator bracket in a unital algebra, $[\one,x]=0$, whence
$J(\one,u,v)=0$.  A basis has the form $\{\one,u,v\}$, so alternation
and trilinearity force $J=0$ identically.
\end{proof}

\begin{theorem}\label{thm:commutator}
For $x=(a,b,c)$ and $y=(d,e,f)$,
\begin{equation}\label{eq:commutator}
[x,y]=(ce-fb)(i+j).
\end{equation}
Consequently, the commutator algebra $\B^-$ is isomorphic to
\[
\R\oplus\mathfrak{aff}(1),
\]
where $\mathfrak{aff}(1)$ is the two-dimensional nonabelian Lie algebra.  In
particular, $\B$ is Lie-admissible and Malcev-admissible.  Furthermore,
\[
Z(\B^-)=\R\one,\qquad [\B^-,\B^-]=\R(i+j).
\]
The Lie algebra $\B^-$ is solvable of derived length two and is not nilpotent.
\end{theorem}

\begin{proof}
Subtracting the two instances of~\eqref{eq:product} gives
\eqref{eq:commutator}.  In the basis $\{\one,j,k\}$, where $k=i+j$, the only
nonzero bracket between basis elements, up to skew-symmetry, is
\begin{equation}\label{eq:aff-bracket}
[j,k]=k.
\end{equation}
This is precisely the direct sum of a one-dimensional center and
$\mathfrak{aff}(1)$.  Equation~\eqref{eq:aff-bracket} proves the Jacobi
identity directly (or identifies the bracket with a known Lie algebra), so
$\B$ is Lie-admissible.  Every Lie algebra satisfies the Malcev identity, hence
$\B$ is also Malcev-admissible.

The same bracket gives $Z(\B^-)=\R\one$ and
$[\B^-,\B^-]=\R k$.  The next derived algebra is $[\R k,\R k]=0$, proving
solvability of derived length two.  On the other hand, the lower central
series stabilizes at $\R k$, because $[j,k]=k$; therefore it is not nilpotent.
\end{proof}

\begin{proposition}[Other admissibility identities]\label{prop:other-identities}
The associator is neither left-symmetric nor right-symmetric, and $\B$ is
not Jordan-admissible.
\end{proposition}

\begin{proof}
The values
\[
\assoc{i}{j}{i}=\one,\qquad
\assoc{j}{i}{i}=-\one+i+2j,\qquad
\assoc{i}{i}{j}=-j
\]
exclude symmetry in the first two and last two arguments, respectively.
Thus neither convention for a pre-Lie identity holds; compare
Section~1.2 of the preprint version of~\cite{GozeRemm2004}.
For Jordan-admissibility, the symmetrized product
$u\circ v=(uv+vu)/2$ would have to satisfy the Jordan identity.
Taking $x=i+2j$, $y=i$, and $x^{\circ2}=x\circ x$, direct multiplication gives
\[
x^{\circ2}\circ(y\circ x)-(x^{\circ2}\circ y)\circ x
=-\one+\tfrac12 i+\tfrac52j\neq0.
\]
\end{proof}

\section{A nontrivial family of deformations}\label{sec:deformations}

The vanishing of derivations and the triviality of the automorphism group
concern symmetries of one multiplication.  They do not imply that nearby
multiplications are isomorphic to it.

\begin{theorem}\label{thm:deformations}
For $t\in\R$, let $\B_t$ have the same unit and basis as $\B$, with
\[
i^2=j^2=-\one,\qquad ij=0,\qquad ji=t(i+j).
\]
Then $\B_1=\B$, and every $\B_t$ is Lie-admissible and absolutely
central simple. For $s,t\in\R$,
\[
\B_s\cong\B_t\quad\Longleftrightarrow\quad s^2=t^2.
\]
For $t\neq0$, one has $\Der(\B_t)=0$ and
$\Aut(\B_t)=\{\mathrm{id}\}$. At $t=0$,
$\Aut(\B_0)\cong O(2)$ and $\Der(\B_0)\cong\mathfrak{so}(2)$.
In particular, $\B$ is not locally rigid among three-dimensional real
unital Lie-admissible algebra laws, with the Euclidean topology on their
structure constants.
\end{theorem}

\begin{proof}
In coordinates the multiplication is
\[
(a,b,c)\mathbin{\cdot_t}(d,e,f)
=(ad-be-cf,\;ae+bd+tce,\;af+tce+cd).
\]
It depends polynomially on $t$, and Proposition~\ref{prop:automatic-lie}
applies.  The operators in~\eqref{eq:nine-ops}, formed with $\cdot_t$, have
vectorization determinant $-t^4$.  One can check this using
\[
L_i^{(t)}=\begin{pmatrix}0&-1&0\\1&0&0\\0&0&0\end{pmatrix},\quad
L_j^{(t)}=\begin{pmatrix}0&0&-1\\0&t&0\\1&t&0\end{pmatrix},\quad
R_i^{(t)}=\begin{pmatrix}0&-1&0\\1&0&t\\0&0&t\end{pmatrix},\quad
R_j^{(t)}=\begin{pmatrix}0&0&-1\\0&0&0\\1&0&0\end{pmatrix}.
\]
The proof of Theorem~\ref{thm:multiplication} therefore establishes absolute
central simplicity for $t\neq0$. At $t=0$, writing $E_{uv}$ for the matrix
units, one obtains
\[
E_{11}=-I-(L_i^{(0)})^2-(L_j^{(0)})^2,\qquad
E_{22}=I+(L_j^{(0)})^2,\qquad E_{33}=I+(L_i^{(0)})^2.
\]
Multiplying $L_i^{(0)}$ and $L_j^{(0)}$ on either side by $E_{11}$ gives
$E_{21},-E_{12},E_{31},-E_{13}$. Their products give $E_{23}$ and $E_{32}$.
Hence the multiplication algebra is again the full matrix algebra over
any extension of $\R$.

For $0<|t|<\sqrt8$, solving $x\cdot_t x=-\one$ gives exactly
\[
\mathcal S_t=\{\pm i,\pm j,\pm w_t\},\qquad
w_t=\frac{2(i+j)-t\one}{\sqrt{8-t^2}}.
\]
Indeed, the equations are
$a^2-b^2-c^2=-1$, $b(2a+tc)=0$, and $c(2a+tb)=0$.
For $bc\neq0$ they force $b=c=-2a/t$, giving the last two roots;
the other cases give $\pm i,\pm j$. If $t^2\geq8$, the same equations
have only $\pm i,\pm j$ as real solutions: for $bc\neq0$ the remaining
equation would be $(t^2-8)a^2=-t^2$, which is impossible.
Thus, for every $t\neq0$, among the roots the ordered zero-product pairs are precisely
$(\varepsilon i,\delta j)$, with $\varepsilon,\delta\in\{1,-1\}$.
To verify completeness when $0<|t|<\sqrt8$, products involving one of
$\pm w_t$ and one of $\pm i,\pm j$ have nonzero scalar coordinate
$\pm2/\sqrt{8-t^2}$; products within a single root line are nonzero, and
$ji=t(i+j)\neq0$.

Suppose $s,t\neq0$. If $\varphi:\B_s\to\B_t$ is an isomorphism,
then it fixes the unit and
must send $i$ to $\varepsilon i$ and $j$ to $\delta j$ by the preceding
zero-product characterization.  Preserving $ji=s(i+j)$ gives
\[
\varepsilon\delta t(i+j)=s(\varepsilon i+\delta j).
\]
Consequently $s=\delta t=\varepsilon t$, so $s=t$ or $s=-t$.
Conversely, the identity identifies equal parameters, and
$\one\mapsto\one$, $i\mapsto-i$, $j\mapsto-j$ defines an
isomorphism $\B_{-t}\to\B_t$ by direct inspection of the table.
For an automorphism $s=t\neq0$, necessarily
$\varepsilon=\delta=1$.

Finally, if $D(i)=a\one+bi+cj$, differentiating $i^2=-\one$ gives
$-2b\one+(2a+tc)i+tcj=0$, so $D(i)=0$ for $t\neq0$.
The corresponding equation for $j$ gives $D(j)=0$, and $D(\one)=0$ as
before. Hence $\Der(\B_t)=0$ for $t\neq0$.

At $t=0$, write $x=(a,u)$ with $u\in\R^2$. The product is
\[
(a,u)\cdot_0(d,v)=(ad-\langle u,v\rangle,av+du).
\]
It is commutative, whereas $\B_t$ is noncommutative for $t\neq0$,
so $\B_0$ is not isomorphic to any of them. Since
$\operatorname{tr}L_{(a,u)}^{(0)}=3a$ and
$L_{\varphi(x)}^{(0)}=\varphi L_x^{(0)}\varphi^{-1}$, an automorphism fixes
the scalar coordinate. Since it also fixes $\one$, it has the form
$(a,u)\mapsto(a,Ou)$. Preserving $(0,u)^2=-\langle u,u\rangle\one$
forces $\langle Ou,Ou\rangle=\langle u,u\rangle$. Polarization gives
$\langle Ou,Ov\rangle=\langle u,v\rangle$, so $O$ is orthogonal.
Conversely this last equality shows directly that every orthogonal $O$
preserves the displayed product. For derivations, differentiating $i^2=j^2=-\one$
gives $D(i)=c j$ and $D(j)=e i$, and differentiating $ij=0$ gives $e=-c$.
For the converse, each such map is $D(a,u)=(0,Ku)$ for a skew-symmetric
matrix $K$. In the derivation rule, the scalar coordinate of
$D(x)y+xD(y)$ is
$-\langle Ku,v\rangle-\langle u,Kv\rangle=0$, and the vector coordinate
is $dKu+aKv=K(du+av)$. These are exactly the coordinates of $D(xy)$.
This proves the assertion about $\mathfrak{so}(2)$. Finally, taking positive $t\to1$, $t\neq1$, gives
arbitrarily close nonisomorphic laws, which proves the claimed failure
of local rigidity.
\end{proof}

\begin{proposition}[Exceptional parameters and quadratic sections]
\label{prop:exceptional}
For $t\neq0$, the unital two-dimensional subalgebras of $\B_t$ are
$\spanR\{\one,i\}$, $\spanR\{\one,j\}$, and $\spanR\{\one,k\}$.
The first two are copies of $\C$, while
\[
\spanR\{\one,k\}\cong\R[X]/(X^2-tX+2).
\]
Consequently the third is isomorphic to $\C$ if $t^2<8$, to the dual numbers
$\R[\epsilon]/(\epsilon^2)$ if $t^2=8$, and to $\R\times\R$ if $t^2>8$.
For $t=0$, every plane $\spanR\{\one,bi+cj\}$ with $(b,c)\neq(0,0)$ is
a copy of $\C$, and $\B_0$ is the Jordan algebra of the bilinear form
$-\langle\,,\,\rangle$ on $\R^2$.
In addition to $\R\one$, the number of one-dimensional real subalgebras
of $\B_t$ is zero for $t^2<8$, one for $t^2=8$, and two for $t^2>8$.
\end{proposition}

\begin{proof}
For $v=bi+cj$, one has
$v\cdot_t v=-(b^2+c^2)\one+tbc k$.
When $t\neq0$, the closure argument of Theorem~\ref{thm:slices} therefore
gives exactly the three planes. The polynomial for $k$ follows from
$k\cdot_t k=-2\cdot\one+tk$. A two-dimensional unital algebra with one
generator and this relation is the displayed associative polynomial
quotient; completing the square gives its three real isomorphism types.
In particular,
\[
(k-t\one/2)^2=(t^2/4-2)\one.
\]
At $t^2=8$ this is a nonzero square-zero element; at $t^2>8$ the elements
$\one/2\pm(k-t\one/2)/\sqrt{t^2-8}$ are nontrivial idempotents.
The square-zero line at $t^2=8$ and the two idempotent lines at $t^2>8$
are subalgebras not containing $\one$. To prove exhaustivity in dimension
one for $t\neq0$, let $x$ be non-scalar with $x^2\in\R x$.
Then $\det[\one,x,x^2]=tbc(b-c)=0$, so $x$ belongs to one of the three
unit-containing planes. A copy of $\C$ has no one-dimensional real
subalgebra other than $\R\one$. In the third plane, normalize the
coefficient of $k$ to write $x=a\one+k$. The condition becomes
$a^2+ta+2=0$, giving exactly zero, one, or two real lines according to
the sign of $t^2-8$. Thus the complete real subalgebra list for $\B_1$
cannot be transferred unchanged to all parameters.

At $t=0$, the relation $v^2=-(b^2+c^2)\one$ gives the stated complex
planes. For $x=(a,u)$, one has
$x^2=2ax-(a^2+\langle u,u\rangle)\one$.
The product is commutative, and this identity implies
$[L_{x^2},L_x]=0$, which is the Jordan identity. Every nonzero subalgebra
of $\B_0$ contains $\one$ by the same quadratic relation with positive
constant term; this also proves the assertion about its one-dimensional
subalgebras. The displayed product for $\B_0$ is precisely the classical
bilinear-form construction in
\cite[Chapter IV, Section 2, example D, p.~101]{Schafer1966}, with
the nondegenerate symmetric form $(u,v)\mapsto-\langle u,v\rangle$
on the two-dimensional space $\R^2$.
\end{proof}

\begin{remark}[Size of the ambient space of laws]
Fixing a unit $\one$ on a three-dimensional vector space leaves the four
products $i^2,ij,ji,j^2$ arbitrary.  The space of these laws is therefore
$\R^{12}$, and all are Lie-admissible.  The group of changes of basis fixing
$\one$ has dimension six, so every orbit has dimension at most six and
none is open in $\R^{12}$. This observation applies to every law in this
ambient space. For $\B$, the stabilizer has zero Lie algebra
$\Der(\B)$, and the orbit dimension is exactly six.
The content of Theorem~\ref{thm:deformations} is the explicit family and
its isomorphism criterion. It provides neither a classification of the
entire ambient space nor a formal or miniversal deformation theory.
\end{remark}

The parameter regimes are summarized in Table~\ref{tab:regimes}.
At $t=0$, the square equations give exactly the circle of roots
$bi+cj$ with $b^2+c^2=1$; at $t\neq0$ the root lists occur in the
proof of Theorem~\ref{thm:deformations}. Although the algebras with
distinct positive parameters are pairwise nonisomorphic, their commutator
Lie algebras are all isomorphic: in the basis $\{\one,j/t,k\}$,
$[j/t,k]_t=k$. Thus the commutator does not recover the parameter.

\begin{table}[ht]
\caption{Real parameter regimes. The axis planes are always copies of $\C$.
The column for $k$ lists the isomorphism type of $\spanR\{\one,k\}$.}
\label{tab:regimes}
\centering
\begin{tabular}{@{}ccccc@{}}
\toprule
Parameter & Roots of $-\one$ & Plane of $k$ & $\Aut(\B_t)$ & $\Der(\B_t)$ \\
\midrule
$t=0$ & $\{bi+cj:b^2+c^2=1\}$ & $\C$ & $O(2)$ & $\mathfrak{so}(2)$ \\
$0<t^2<8$ & $\pm i,\pm j,\pm w_t$ & $\C$ & $\{\mathrm{id}\}$ & $0$ \\
$t^2=8$ & $\pm i,\pm j$ & $\R[\epsilon]/(\epsilon^2)$ & $\{\mathrm{id}\}$ & $0$ \\
$t^2>8$ & $\pm i,\pm j$ & $\R\times\R$ & $\{\mathrm{id}\}$ & $0$ \\
\bottomrule
\end{tabular}
\end{table}

\clearpage
\section{Concluding remarks}

The polynomial $bc(b-c)$ describes both the quadratic locus and the
obstruction to power-associativity in $\B$. Its zero set consists of three
complex subalgebras whose nonzero elements have unique two-sided inverses.
Outside it, one element generates $\B$. The one-sided inverse criteria
further separate inversion from cancellation, even though the algebra
is absolutely central simple.

The family $\B_t$, classified here up to real isomorphism, shows that
trivial automorphisms and derivations coexist with nontrivial deformations.
Its parameters $0$ and $t^2=8$ account for changes in symmetries and
quadratic sections. These explicit calculations provide counterexamples
to associative intuitions and reproducible examples for structural
computations in nonassociative algebra.

\section*{Computational verification}
Exact symbolic calculations were checked with SymPy 1.14.0.
The ancillary file \nolinkurl{verify_certificates.py} reconstructs the
multiplication table and checks the polynomial identities and finite
certificates with exact arithmetic. It runs independently of the typesetting
files. The classification, positivity and convergence arguments are given
in the text; the computation supplements those proofs.
Responsibility for the article's contents rests with the author.

\small
\bibliographystyle{amsplain}
\bibliography{three_dimensional_lie_admissible_algebra}

\end{document}